\documentclass{amsart}

\usepackage{amssymb}
\usepackage[hidelinks]{hyperref}

\newtheorem{theorem}{Theorem}[section]
\newtheorem{proposition}[theorem]{Proposition}
\newtheorem{lemma}[theorem]{Lemma}

\theoremstyle{definition}
\newtheorem{definition}[theorem]{Definition}
\newtheorem{example}[theorem]{Example}

\theoremstyle{remark}
\newtheorem{remark}[theorem]{Remark}

\begin{document}

\title[Projective Lifting Theorem and Quasi-Galois Points]{Projective Lifting Theorem, Complex Reflection Groups, and Quasi-Galois Points for Fermat Varieties}

\author[K.\ Miura]{Kei MIURA}
\address{Department of Applied Science, Yamaguchi University, 2-16-1 Tokiwadai, Ube, \linebreak Yamaguchi 755-8611, Japan}
\email{kmiura@yamaguchi-u.ac.jp}

\author[T.\ Takahashi]{Takeshi TAKAHASHI}
\address{Education Center for Engineering and Technology, Faculty of Engineering, Niigata University, 8050 Ikarashi-2-no-cho Nishi-ku, Niigata 950-2181, Japan}
\email{takeshi@eng.niigata-u.ac.jp}



\begin{abstract}
	We investigate birational automorphisms of hypersurfaces that preserve the projection from a point. Let $V \subset \mathbb{P}^{n+1}$ be an irreducible hypersurface of degree $d \ge 3$ with $\dim \mathrm{Sing}(V) \le n-2$, and let $\pi_P : V \dashrightarrow \mathbb{P}^n$ denote the projection from a point $P \in \mathbb{P}^{n+1} \setminus \mathrm{Sing}(V)$. Assume moreover that if $d = 3$ then $P \notin V$. We prove that any $\sigma \in \mathrm{Bir}(V)$ satisfying $\pi_P \circ \sigma = \pi_P$ extends uniquely to a projective transformation of $\mathbb{P}^{n+1}$.

	By combining the above projective lifting theorem with the theory of complex reflection groups, we investigate Galois and quasi-Galois points for the Fermat variety $F_d^n \subset \mathbb P^{n+1}$ of dimension $n$ and degree $d$. We show that, for $d \ge 4$, $F_d^n$ has exactly $n+2$ Galois points and, in addition, $d(n+2)(n+1)/2$ quasi-Galois points. When $d=3$, $F_3^n$ has exactly $n+2$ outer Galois points.
\end{abstract}

\maketitle

\section{Introduction and Main Results}

Let $k$ be the field of complex numbers ${\mathbb C}$. We fix it as the ground field of our discussion. Let $V \subset {\mathbb P}^{n + 1}$ be an irreducible hypersurface of degree $d$, and $k(V)$ the function field of $V$. Let ${\rm Bir}(V)$ be the group of birational automorphisms of $V$, that is, birational self-maps of $V$. Let $P$ be a point of ${\mathbb P}^{n + 1}$. Consider a projection $\pi_P : V \dashrightarrow {\mathbb P}^n$, which is the restriction of the projection ${\mathbb P}^{n + 1} \dashrightarrow {\mathbb P}^n$ with the center $P$. Then we obtain the field extension induced by $\pi_P$, i.e., $\pi_P^* : k({\mathbb P}^n) \hookrightarrow k(V)$. Putting $K_P = \pi_P^* (k({\mathbb P}^n))$, we recall the notion of a Galois point due to Yoshihara. 

\begin{definition}
	The point $P$ is called a \emph{Galois point} for $V$ if the field extension $k(V)/K_P$ is Galois. Furthermore, a Galois point is said to be \emph{inner} (resp. \emph{outer}) if $P \in V$ (resp. $P \in {\mathbb P}^{n + 1} \setminus V$). The group $G_P := {\rm Gal} (k(V)/K_P)$ is called the \emph{Galois group} at $P$.  
\end{definition}

To generalize the notion of Galois points, we introduce the notion of \emph{quasi-Galois points}. Quasi-Galois points were introduced by Fukasawa, Miura and Takahashi in \cite{fmt1} and \cite{fmt2}, where a systematic and detailed study of quasi-Galois points on nonsingular plane curves was carried out. 

First, we define the following set. 

\begin{definition}
	For a point $P \in {\mathbb P}^{n + 1}$, we define
	\[
		G[P] := \{ \tau \in {\rm Bir}(V) \, | \,  \pi_P \circ \tau = \pi_P \}. 
	\]
\end{definition}

\begin{remark}
	The set $G[P]$ is a subgroup of ${\rm Bir}(V)$, and its order divides the degree of $\pi_P$, since $G[P]$ acts faithfully on a general fiber of $\pi_P$.
\end{remark}

A quasi-Galois point is defined as follows. 

\begin{definition}
	The point $P$ is called a \emph{quasi-Galois point} if $|G[P]| \geq 2$. 
\end{definition}

\begin{remark}
	The condition $|G[P]| = \deg \pi_P$ holds if and only if $P$ is Galois. 
\end{remark}

The purpose of this article is to study properties of birational automorphisms of hypersurfaces compatible with projections, that is, elements of $G[P]$. As an application, we investigate Galois and quasi-Galois points for Fermat varieties from the viewpoint of complex reflection groups. The Fermat varieties are defined by 
\[
	F_d^n : x_0^d + x_1^d + \cdots + x_{n + 1}^d = 0 \subset {\mathbb P}^{n + 1}.
\]

To investigate quasi-Galois points, it is important to understand the elements of $G[P]$. If every element of $G[P]$ extends to a projective transformation of the ambient projective space, then one can apply methods from linear algebra and projective geometry instead of dealing directly with birational transformations. This considerably simplifies the study of quasi-Galois points.

Several results related to the projective extensibility of automorphisms and birational transformations associated with Galois points are known. For nonsingular quartic surfaces in $\mathbb P^3$, Yoshihara proved the corresponding result for inner Galois points \cite{YoshiharaQuartic}. His argument is based on the fact that a nonsingular quartic surface is a non-ruled minimal surface together with the complete linear system defining the embedding.  Moreover, the same argument works without essential modification for both inner and outer quasi-Galois points, and more generally for nonsingular surfaces in $\mathbb P^3$ of degree at least four. For nonsingular hypersurfaces, except for nonsingular quartic surfaces in $\mathbb P^3$, the classical theorem of Matsumura and Monsky \cite{MM} shows that every automorphism is induced by a projective transformation. Thus, once a birational transformation is known to be an automorphism, its projective extensibility follows from their theorem. Related results for nonsingular hypersurfaces with Galois points were obtained by Yoshihara in \cite{yoshi2}. However, as we explain in Remark~\ref{rem:yoshi2}, the argument given there does not appear to establish the regularity of all birational transformations belonging to Galois points. We also note that, for outer Galois points, related results were obtained by Yoshihara in \cite{Yoshihara2013}. In particular, the arguments in Representation~1, Remark~2.1, and Representation~2 of \cite{Yoshihara2013} show that the corresponding Galois group admits a projective representation. Moreover, the same argument applies without essential modification to outer quasi-Galois points. For normal hypersurfaces with Galois points, Fukasawa and Takahashi \cite{FukasawaTakahashi2014} proved that every element of the Galois group is induced by a projective transformation. Their argument, however, essentially uses the assumption that the point is Galois and does not directly apply to quasi-Galois points. These considerations naturally lead to the following question: given $\sigma\in G[P]$, when does $\sigma$ extend to a projective transformation of the ambient projective space? The following projective lifting theorem gives an affirmative answer for a large class of hypersurfaces, without assuming that $P$ is a Galois point.

\begin{theorem}\label{projtransf}
	Let $V \subset \mathbb{P}^{n+1}$ be an irreducible hypersurface of degree $d \ge 3$, and let $P \in \mathbb{P}^{n+1}$ be a point. Let $\pi_P : V \dashrightarrow \mathbb{P}^n$ be the projection from the center $P$. Assume that $\dim \mathrm{Sing}(V) \le n-2$ and $P \notin \mathrm{Sing}(V)$. Assume further that either
	\begin{enumerate}
    	\item[\normalfont (i)] $d \ge 4$, or
    	\item[\normalfont (ii)] $d = 3$ and $P \notin V$.
	\end{enumerate}
	Then, for any $\sigma \in \mathrm{Bir}(V)$ satisfying $\pi_P \circ \sigma = \pi_P$, there exists a unique projective transformation $T \in \mathrm{Aut}(\mathbb{P}^{n+1})$ such that $T(V) = V$ and $T|_V = \sigma$.
\end{theorem}

The assumption that the ground field $k$ is the field of complex numbers is not essential for Theorem~\ref{projtransf}. The same statement holds over any algebraically closed field of characteristic zero.

\begin{remark}\label{rem:yoshi2} 
	In \cite{yoshi2}, Yoshihara stated that every element of the Galois group associated with a Galois point of a nonsingular hypersurface is a morphism. Although \cite{yoshi2} works under the standing assumption $d\ge 4$, the proof of this assertion does not appear to use the condition $d\ge 4$ in an essential way. Indeed, consider the nonsingular cubic surface 
	\[ 
		XW^2+Y^2W+X^3+Y^3+Z^3=0. 
	\] 
	The point $P=(0:0:0:1)$ is an inner Galois point. The corresponding Galois group is generated by the birational involution \[ (X:Y:Z:W)\longmapsto (X:Y:Z:-W-Y^2/X), \] which is not a morphism (see \cite{m}). If the above argument in \cite{yoshi2} were applicable, this involution would necessarily be a morphism. Thus, the argument does not appear to establish the claimed morphism property in this situation. This example also illustrates why the assumption $P\notin V$ is necessary in Theorem~\ref{projtransf}(ii). 
\end{remark}

Our second main result is a complete determination of the numbers of Galois and quasi-Galois points for Fermat varieties. Applying Theorem~\ref{projtransf}, we may regard elements of $G[P]$ as projective transformations. Combining this viewpoint with the theory of complex reflection groups, we obtain the following theorem.

\begin{theorem} \label{thm:FermatGalois}
 	For $F_d^n$ $(d \geq 4, n \geq 1)$, we have the following. 
 	\begin{enumerate}
  		\item[$(1)$] $F_d^n$ has $n + 2$ Galois points whose Galois groups are isomorphic to the cyclic group of order $d$.  
  		\item[$(2)$] $F_d^n$ has $d (n + 2)(n + 1)/2$ quasi-Galois points for which $G[P]$ is cyclic of order two.
 	\end{enumerate}
 	In particular, these are all the quasi-Galois points for $F_d^n$.
\end{theorem} 

\begin{remark} \label{rem:d=3_Fermat}
	When $d=3$, the projection from any point $P\in F_3^n$ induces a quadratic extension of function fields, and hence every such point is a (trivial) inner Galois point. However, in this case, the involution generating $G[P]$ does not necessarily extend to a projective transformation of the ambient projective space. See \cite{m} for details. On the other hand, $F_3^n$ has exactly $n+2$ outer Galois points. 
\end{remark}

For the Fermat curve $F_d^1$, the numbers of Galois points and quasi-Galois points were already determined in \cite{fmt1}, \cite{m-y1}, \cite{m-y2}, and \cite{yoshi1}.  
Moreover, the number of Galois points for the Fermat variety $F_d^n$ with $n \ge 2$, as stated in Theorem~\ref{thm:FermatGalois}(1) and Remark~\ref{rem:d=3_Fermat}, is also known. In fact, Yoshihara proved that $F_d^n$ has exactly $n+2$ outer Galois points (see \cite{yoshi2}). 
For the Fermat quartic surface, Miura and Taki \cite{MT} independently obtained, around the same time as the present work and by a different approach, the complete determination of quasi-Galois points, showing that they consist of $4$ Galois points and $24$ quasi-Galois points. Apart from the Fermat curve and the Fermat quartic surface, the number of quasi-Galois points for Fermat varieties has not been determined previously. 
The novelty of the present paper lies not only in the complete determination of quasi-Galois points for $F_d^n$, but also in the method. By combining projective-transformational properties of elements of $G[P]$ with the theory of complex reflection groups, we obtain a unified and conceptual approach to Galois points and quasi-Galois points. In particular, this approach allows us to interpret quasi-Galois points as reflections and thereby determine all quasi-Galois points for $F_d^n$.

\section{Proof of Theorem~\ref{projtransf}}

In order to prove Theorem~\ref{projtransf}, we first establish several lemmas concerning the behavior of rational functions under linear substitutions of the last variable.  These results will be used later to control the form of the automorphism $\sigma^*$ on the function field $k(V)$.

Throughout this section, we assume that $k$ is an algebraically closed field of characteristic zero. Let $k(x_1, \dots, x_n)$ denote the field of rational functions in $n$ variables over $k$.

\begin{lemma} \label{Lemma f=0}
	Let $n\ge2$. Let $f = f(x_1, \dots, x_n) \in k(x_1, \dots, x_n)$.  
 	If
  	\[
   		f\biggl( x_1, \dots, x_{n-1}, \alpha_0 + \sum_{i=1}^{n-1} \alpha_i x_i \biggr) = 0
  	\]
 	for a general point $(\alpha_0, \dots, \alpha_{n-1}) \in k^n$, then $f(x_1, \dots, x_n)=0$. 
\end{lemma}

\begin{proof}
	Let $H_{\alpha} := X_n - \sum_{i=0}^{n-1} \alpha_i X_i$. Then $(H_{\alpha}=0) \subset \mathbb{P}^n$ is a general hyperplane. By assumption, $f|_{H_{\alpha}} = 0$ in $k(H_{\alpha})$. Let $A/B$ be the homogenization of $f$, where $A, B \in k[X_0, \dots, X_n]$ are homogeneous polynomials of the same degree. Since $f|_{H_{\alpha}} = 0$, the polynomial $H_{\alpha}$ divides $A$. If $A \neq 0$, then the number of irreducible factors of $A$ is finite. As $(\alpha_0, \dots, \alpha_{n-1})$ is general, this implies that $A = 0$, and hence $f=0$.  
\end{proof}

\begin{lemma} \label{Lemma f in k}
	Let $n\ge2$. Let $f = f(x_1, \dots, x_n) \in k(x_1, \dots, x_n)$. If
	\[
		f\biggl( x_1, \dots, x_{n-1}, \alpha_0 + \sum_{i=1}^{n-1} \alpha_i x_i \biggr) \in k(x_1, \dots, x_{n-1})
	\]
	is a constant function for a general point $(\alpha_0, \dots, \alpha_{n-1}) \in k^n$, then $f(x_1, \dots, x_n)$ is also constant.
\end{lemma}

\begin{proof}
	For each $\alpha := (\alpha_0, \dots, \alpha_{n-1})$, the constant value
	\[
		f\bigl( x_1, \dots, x_{n-1}, \alpha_0 + \sum_{i=1}^{n-1} \alpha_i x_i \bigr)
	\]
	depends only on $\alpha$.  Define a function 
	\[
		\tilde{f} : U \rightarrow k, \qquad \alpha \mapsto f\bigl( x_1, \dots, x_{n-1}, \alpha_0 + \sum_{i=1}^{n-1} \alpha_i x_i \bigr),
	\]
	where $U$ is a Zariski-open subset of $k^n$ on which $\tilde{f}$ is well-defined. Then $\tilde{f}$ is a regular function on $U$. Compute the partial derivatives:  
	\[
		\frac{\partial \tilde{f}}{\partial \alpha_0} = \frac{\partial f}{\partial x_n} \bigg|_{x_n = \alpha_0 + \sum_{i=1}^{n-1} \alpha_i x_i}, \qquad
		\frac{\partial \tilde{f}}{\partial \alpha_i} = x_i \frac{\partial f}{\partial x_n} \bigg|_{x_n = \alpha_0 + \sum_{i=1}^{n-1} \alpha_i x_i} \quad (1 \le i \le n-1).
	\]
	We have
	\[
		\frac{\partial\tilde f}{\partial\alpha_i} = x_i \frac{\partial\tilde f}{\partial\alpha_0} \qquad (1\le i\le n-1).
	\]
	Since $\tilde f$ is constant as a function of $x_1,\dots,x_{n-1}$, the partial derivatives $\partial\tilde f/\partial\alpha_0$ and $\partial\tilde f/\partial\alpha_i$ are constant in $x_1,\dots,x_{n-1}$. Therefore
	\[
		\frac{\partial\tilde f}{\partial\alpha_0}=0, \text{ and hence } \, \frac{\partial\tilde f}{\partial\alpha_i}=0 \quad (i=1,\dots,n-1).
	\]
	Thus $\tilde f$ is constant in $\alpha_0,\dots,\alpha_{n-1}$. Now apply Lemma~\ref{Lemma f=0} to the function \linebreak $f(x_1,\dots,x_n) - \tilde{f}$. We have that $f = \tilde{f}$. Hence $f$ is constant.
\end{proof}

\begin{lemma} \label{Lemma f is linear}
	Let $n\ge2$. Let $f = f(x_1, \dots, x_n) \in k(x_1, \dots, x_n)$. Suppose that for a general point $(\alpha_0, \dots, \alpha_{n-1}) \in k^n$, the rational function
	\[
		f\!\left( x_1, \dots, x_{n-1}, \alpha_0 + \sum_{i=1}^{n-1} \alpha_i x_i \right)
	\]
	is a linear combination of $1, x_1, \dots, x_{n-1}$ over $k$. Then $f(x_1,\dots,x_n)$ itself is a linear combination of $1, x_1, \dots, x_n$ over $k$.
\end{lemma}

\begin{proof}
	Put
	\[
		\tilde{f} := f\!\left(x_1, \dots, x_{n-1}, \alpha_0 + \sum_{i=1}^{n-1} \alpha_i x_i \right) 
			= c_0 + c_1 x_1 + \dots + c_{n-1} x_{n-1},
	\]
	where $c_i \in k$.  
	For a general choice of $\alpha = (\alpha_0, \dots, \alpha_{n-1})$, the coefficients $c_i$ depend on $\alpha$. Each $c_i$ is a regular function on some Zariski-open subset $U \subset k^n$. Differentiating $\tilde{f}$ with respect to $\alpha_0$ yields
	\[
		\frac{\partial \tilde{f}}{\partial \alpha_0}
		= \frac{\partial f}{\partial x_n} \Bigg|_{x_n = \alpha_0 + \sum_{i=1}^{n-1} \alpha_i x_i}
		= \frac{\partial c_0}{\partial \alpha_0} + \sum_{j=1}^{n-1} \frac{\partial c_j}{\partial \alpha_0} \, x_j.
	\]
	Similarly, for $i = 1, \dots, n-1$,
	\[
		\frac{\partial \tilde{f}}{\partial \alpha_i}
		= x_i \, \frac{\partial f}{\partial x_n} \Bigg|_{x_n=\alpha_0 + \sum_{j=1}^{n-1} \alpha_j x_j}
		= \frac{\partial c_0}{\partial \alpha_i} + \sum_{j=1}^{n-1} \frac{\partial c_j}{\partial \alpha_i}\, x_j.
	\]
	We have that ${\partial \tilde{f}}/{\partial \alpha_i} = x_i \, {\partial \tilde{f}}/{\partial \alpha_0}$ $(i=1,\dots,n-1)$. Comparing coefficients of $1, x_1, \dots, x_{n-1}$, we obtain
	\[
		\frac{\partial c_j}{\partial \alpha_i} = 0 \quad (i \ne j), \qquad
		\frac{\partial c_i}{\partial \alpha_i} = \frac{\partial c_0}{\partial \alpha_0} \quad (i=1,\dots,n-1).
	\]
	We see that the function $e := {\partial c_i}/{\partial \alpha_i} = {\partial c_0}/{\partial \alpha_0}$ is constant on $U$. We may write $c_i = d_i + e \alpha_i$ $(i=0,\dots,n-1)$, for some constants $d_0,\dots,d_{n-1} \in k$. Then
	\[
		\tilde{f} = (d_0 + e\alpha_0) + \sum_{i=1}^{n-1} (d_i + e\alpha_i)\, x_i
				= \left( d_0 + \sum_{i=1}^{n-1} d_i x_i \right) + e\left( \alpha_0 + \sum_{i=1}^{n-1} \alpha_i x_i \right).
	\]
	Finally, applying Lemma~\ref{Lemma f=0} to $f(x_1, \dots, x_n) - \left(d_0 + \sum_{i=1}^{n-1} d_i x_i + e x_n\right)$, we conclude that $f = d_0 + \sum_{i=1}^{n-1} d_i x_i + e x_n$.
\end{proof}

Let $V \subset \mathbb{P}^{n+1}$ be an irreducible hypersurface of degree $d \ge 3$ with $n \ge 2$. Let $P \in \mathbb{P}^{n+1}$ be a point, and let $\pi_P : V \dashrightarrow \mathbb{P}^n$ denote the projection from $P$. Let ${\rm Sing}(V)$ be the singular locus of $V$, where we set $\dim {\rm Sing}(V) = -1$ when ${\rm Sing}(V)=\emptyset$. Assume that $\dim {\rm Sing}(V) \le n-2$ and $P \notin {\rm Sing}(V)$. Then we have the following:

\begin{lemma}[Lemma~2.5 and Proof of Theorem~1.3 in \cite{FukasawaTakahashi2014}] \label{Lemma general hyperplane section}
	For a general hyperplane $H$ passing through the point $P$, the hyperplane section $V \cap H \subset H \cong \mathbb{P}^n$ is an irreducible hypersurface of degree $d$ satisfying $\dim {\rm Sing}(V \cap H) \le n-3$ and $P \notin {\rm Sing}(V \cap H)$. Moreover, if $\sigma \in {\rm Bir}(V)$ satisfies $\pi_P = \pi_P \circ \sigma$, then $\sigma|_{V \cap H} \in {\rm Bir}(V \cap H)$ and $\pi_P|_{V \cap H} = \pi_P|_{V \cap H} \circ \sigma|_{V \cap H}$.
\end{lemma}

We now have all the necessary preparations to prove Theorem~\ref{projtransf}. We prove Theorem~\ref{projtransf}.

We first show the uniqueness. Let $T_1$ and $T_2$ be projective transformations of $\mathbb{P}^{n+1}$ such that $T_i(V) = V$ and $T_i|_V = \sigma$ for $i = 1,2$. Then the fixed locus of $T_1 \circ T_2^{-1}$ in $\mathbb{P}^{n+1}$ is a union of linear subspaces. Since $T_1|_V = T_2|_V$, this fixed locus contains $V$. As $V$ is an irreducible hypersurface and not a linear subspace, we conclude that $T_1 = T_2$.

We now prove the existence by induction on $n = \dim V$.

If $n = 1$, then $V \subset \mathbb{P}^2$ is a nonsingular plane curve of degree $d \ge 3$. If $d \ge 4$, the assertion is classical (see \cite{Chang1978} or Exercises~17–18 in Appendix~A of \cite{ACGH1985}). If $d = 3$ and $P \notin V$, the assertion follows immediately, since the embedding $V \hookrightarrow \mathbb{P}^2$ is given by the complete linear system associated with the divisor $D := V \cap l$, where $l$ is a line passing through $P$, and the condition $\pi_P \circ \sigma = \pi_P$ forces $\sigma^{*}D = D$ (see Remark~1 of \cite{Yoshihara2007}). 

Assume that $n \ge 2$. We may assume that $P = (0:\dots:0:1)$ and that
\[
	\pi_P : (X_0:\dots:X_{n+1}) \mapsto (X_0:\dots:X_n).
\]
Let $x_i$ be the rational functions on $V$ defined by $x_i = X_i/X_0$ for $i = 1,\dots,n+1$. Then the extension $k(V) / \pi_P^*(k(\mathbb{P}^n))$ induced by $\pi_P$ is written as
\[
	k(x_1,\dots,x_n)(x_{n+1}) \, / \, k(x_1,\dots,x_n).
\]
Since $\pi_P \circ \sigma = \pi_P$, the induced automorphism $\sigma^*$ of $k(V)$ satisfies $\sigma^*(x_i) = x_i$ for $i = 1,\dots,n$. Write
\[
	\sigma^* x_{n+1} = B_0 + B_1 x_{n+1} + \dots + B_r x_{n+1}^r \in k(x_1,\dots,x_n)(x_{n+1}),
\]
where each $B_i \in k(x_1,\dots,x_n)$ and $r$ is strictly less than the degree of the extension.

We show that $\sigma^* x_{n+1}$ is a linear combination of $1, x_1, \dots, x_{n+1}$ over $k$. Let $H$ be a general hyperplane passing through $P$. We may assume that $H$ is defined by $X_n = \sum_{i=0}^{n-1} \alpha_i X_i$. By Lemma~\ref{Lemma general hyperplane section} and the induction hypothesis, there exists a projective transformation $T' \in \operatorname{Aut}(H) \cong \operatorname{PGL}_{n+1}(k)$ such that $T'(V \cap H) = V \cap H$ and $T'|_{V \cap H} = \sigma|_{V \cap H}$. A system of homogeneous coordinates on $H$ is given by $(X_0:\dots:X_{n-1}:X_{n+1})$. Since $\pi_P|_{V \cap H} = \pi_P|_{V \cap H} \circ \sigma|_{V \cap H}$, the transformation $T'$ is represented by a matrix of the form
\[
	T' = 
		\begin{pmatrix}
			1 & & & \\
			& \ddots & & \\
			& & 1 & \\
			* & \cdots & * & *
		\end{pmatrix}.
\]
Hence $(\sigma|_{V \cap H})^* x_{n+1}$ is a linear combination of $1, x_1,\dots,x_{n-1}, x_{n+1}$ over $k$. Since $(\sigma^* x_{n+1})|_{V \cap H} = (\sigma|_{V \cap H})^*(x_{n+1}|_{V \cap H})$, substituting $x_n = \alpha_0 + \sum_{i=1}^{n-1} \alpha_i x_i$ into $\sigma^* x_{n+1} = B_0 + \sum_{i=1}^r B_i x_{n+1}^i$, we obtain
\[
	B_0|_{V \cap H} + \sum_{i=1}^r B_i|_{V \cap H} \, x_{n+1}^i = c_0 + c_1 x_1 + \dots + c_{n-1} x_{n-1} + c_{n+1} x_{n+1},
\]
where $c_j \in k$. Since the degree of the extension induced by \(\pi_P|_{V \cap H}\) is greater than \(r\), we obtain $B_2|_{V \cap H} = \dots = B_r|_{V \cap H} = 0$. By Lemma~\ref{Lemma f=0}, it follows that $B_2 = \dots = B_r = 0$. Moreover, $B_0|_{V \cap H}$ is a linear combination of $1, x_1,\dots,x_{n-1}$ over $k$, and $B_1|_{V \cap H}$ is constant. By Lemmas~\ref{Lemma f in k} and \ref{Lemma f is linear}, $B_0$ is a linear combination of $1, x_1,\dots,x_n$ over $k$, and $B_1$ is constant. Thus $\sigma^* x_{n+1}$ is a linear combination of $1, x_1,\dots,x_{n+1}$ over $k$. 

Therefore there exists a projective transformation $T \in \operatorname{Aut}(\mathbb{P}^{n+1})$ such that $T(V) = V$ and $T|_V = \sigma$. This completes the proof of Theorem~\ref{projtransf}.

\section{Complex Reflection Groups}

In this section, we recall the notion of complex reflection groups in brief (cf. \cite{dol}, \cite{l-t}, \cite{s-t}). Let $V$ be a vector space of dimension $n$ over the field of complex numbers ${\mathbb C}$, and $(-,-)$ be a positive definite hermitian form. Let $a \in V$ ($a \ne 0$) be a vector and $\alpha$ an $m$-th root of unity $(\alpha \ne 1)$. 

\begin{definition}
	For $v \in V$, 
	\[
		r_{a, \alpha} (v) := v - (1 - \alpha) \frac{(v, a)}{(a,a)} a.
	\]
	We call $r_{a, \alpha}$ a \emph{reflection}. 
\end{definition}

\begin{example}
	For example, when $n = 3$, let $a$ be a vector ${}^t\!(1,0,0)$ and $v$ be ${}^t\!(x,y,z)$, then $r_{a, \alpha}$ is given by
	\[
		r_{a, \alpha} (v) = v - (1 - \alpha) \frac{x}{1} a = 
			\begin{pmatrix}
				\alpha x \\
        			y \\
        			z 
			\end{pmatrix}
        =	\begin{pmatrix}
        		\alpha & 0 & 0 \\
   				0 & 1 & 0 \\
   				0 & 0 & 1
       	 	\end{pmatrix}
  		  	\begin{pmatrix}
  		  		x \\
  				y \\
  				z 
  		  	\end{pmatrix},
  	\]
	where $(-,-)$ is the standard hermitian inner product. Namely, $r_{a, \alpha}$ is given by the above matrix. 
\end{example}
 
Actually, we have another definition of reflections. 

\begin{definition}
	Let $M \in {\rm GL} (n, {\mathbb C})$. Then $M$ is called a \emph{reflection} if ${\rm rank} (M - E) = 1$ and $M^k = E$ for some integer $k \geq 2$, where $E$ is the identity matrix. 
\end{definition}

That is, the eigenvalues of a reflection are $\{ \alpha, 1,\ldots, 1 \}$, 
for some $\alpha \ne 1$. 

\begin{proposition}[{cf. \cite[Chapter 1, Section 2]{l-t}}]
	We have the following properties. 
	\begin{enumerate}
  		\item[$(1)$] $(r_{a, \alpha})^m = {\rm id}$. 
  		\item[$(2)$] For any $\lambda \in {\mathbb C} \setminus \{ 0 \}$, $r_{a, \alpha} = r_{\lambda a, \alpha}$. 
  		\item[$(3)$] For any $y \in \langle a \rangle^{\perp}$, $r_{a, \alpha} (y) = y$. 	
	\end{enumerate}
\end{proposition}

Our strategy is as follows. If there exists a quasi-Galois point $P$, then we obtain the group belonging to $P$, i.e., $G[P]$. Roughly speaking, $G[P]$ is generated by a reflection. Actually, we have already proved the following theorem in the case where $V$ is a nonsingular plane curve of degree $d \geq 4$. 

\begin{theorem}[{\cite[Theorem 2.3]{fmt1}}]
	The group $G[P]$ is a cyclic group. Furthermore, for an integer $n \geq 2$, $n$ divides $|G[P]|$ if and only if there exists a linear coordinate change $\phi$ such that 
 	\begin{enumerate}
  		\item[$(1)$] $\phi (P) = (1:0:0)$,
  		\item[$(2)$] there exists an element $\sigma \in G[\phi (P)] \subset \mathrm{Bir} (\phi (V))$ which is represented by the matrix 
  			$A_\sigma = 
  				\begin{pmatrix}
  					\zeta & 0 & 0 \\
   					0 & 1 & 0 \\
   					0 & 0 & 1 
  				\end{pmatrix}$, where $\zeta$ is a primitive $n$-th root of unity, 
  		\item[$(3)$] the defining equation of $\phi (V)$ is given by 
  			$\sum_i G_{d - n i} (x_1, x_2) x_0^{n i} = 0$, where $G_{d - n i}$ is a homogeneous polynomial of degree $d - n i$. 
 	\end{enumerate}
\end{theorem}

As in the argument of the above theorem, under the assumption of Theorem~\ref{projtransf}, namely when every element of $G[P]$ is realized as the restriction of a projective transformation, we see that $G[P]$ is generated by a reflection. Hence, we infer that quasi-Galois points can be regarded as a geometric interpretation of reflections. Conversely, by applying the theory of reflections, we investigate quasi-Galois points. In particular, by counting reflections, we obtain the number of quasi-Galois points.

\begin{definition}
	A \emph{complex} (or \emph{unitary}) \emph{reflection group} is a finite subgroup of the unitary group $U(V)$ that is generated by reflections.
\end{definition}

\begin{remark}
	Every finite subgroup of ${\rm GL} (V)$ that is generated by reflections is a unitary reflection group with respect to some positive definite hermitian form. 
\end{remark}

Now, we recall Shephard and Todd's classification \cite{s-t} (or \cite{dol}, \cite{l-t}). We use their notation. There are three infinite families of reflection groups and these are numbered $1$, $2$ and $3$. In addition there are $34$ primitive groups of rank greater than one numbered from $4$ to $37$. Hereafter we focus on $G(m, p, n)$. 

\begin{definition}
	Let $m$, $p$ and $n$ be integers such that $p|m$ and $\zeta_m$ be a primitive $m$-th root of unity. The groups $G(m, p, n)$ are subgroups of ${\rm GL}(n, {\mathbb C})$ consisting of the following matrices, 
	\begin{enumerate}
		\item[$\bullet$] the entries are either $0$ or a power of $\zeta_m$, 
  		\item[$\bullet$] there is exactly one non-zero entry in each row and column, 
  		\item[$\bullet$] the $(m/p)$-th power of the product of all non-zero entries is $1$. 
	\end{enumerate}
\end{definition}

Alternatively, we shall use the following wreath product description of $G(m,1,n)$.
\begin{definition}
	\[
		G(m,1,n) = {\mathbb Z}_m \wr S_n
	\]
\end{definition}

Namely, 
\[
	{\mathbb Z}_m \wr S_n = \left\{ ((a_1, a_2, \ldots, a_n),\sigma) \, | \, 0 \leq a_i \leq m-1, \, \sigma \in S_n \right\}.
\]
In particular, the order of $G(m,1,n)$ is $m^n n!$.

\section{Fermat Varieties}

In this section, we give some results on Galois and quasi-Galois points for Fermat varieties from the viewpoint of reflection groups. Throughout this section, we assume $d\ge4$ unless otherwise stated. Let $F_d^n \subset {\mathbb P}^{n + 1}$ be the Fermat variety $x_0^d + x_1^d + \cdots +  x_{n + 1}^d = 0$.  Taking a primitive $d$-th root of unity $\zeta = \zeta_d$, we have the following automorphisms of $F_d^n$: 
\[
	g_{i_0 i_1 \ldots i_n} : (x_0: x_1: \ldots: x_n: x_{n+1}) \mapsto (\zeta^{i_0} x_0: \zeta^{i_1}x_1: \ldots: \zeta^{i_n}x_n: x_{n+1}),
\]
where $0 \le i_0, i_1, \ldots, i_n < d$. They form an abelian group ${\mathbb Z}_d^{n + 1}$ of order $d^{n + 1}$. Also the symmetric group $S_{n + 2}$ acts on $F_d^n$ by permuting $x_0, x_1, \ldots, x_n, x_{n+1}$. Here, we recall the following well-known theorem (e.g.\ \cite{shioda}, \cite{tz}). 

\begin{theorem}
	Assume $(n, d) \ne (2, 4)$. Then ${\rm Aut} (F_d^n) \cong {\mathbb Z}_d^{n + 1} \rtimes S_{n + 2}$.  In particular, the order of ${\rm Aut} (F_d^n)$ is $(n + 2)! d^{n + 1}$. 
\end{theorem}

\begin{remark}
	We note that ${\rm Aut} (F_d^n)$ acts on $F_d^n$ as projective transformations. In case $(n, d) = (2, 4)$, $F_4^2$ is a $K3$ surface. It is known that there exists an automorphism of quartic surfaces which is not a restriction of a projective transformation. However, by Theorem~\ref{projtransf},  we can assume that any element of $G[P]$ is the restriction of a projective transformation. That is, it is enough to consider only ${\rm Lin (F_4^2) } \cong {\mathbb Z}_4^{3} \rtimes S_4$. 
\end{remark}

By \cite{l-t}, the polynomial $x_0^d + x_1^d + \cdots + x_{n+1}^d$ is invariant under the action of $G(d,1,n+2)$. Moreover, 
\[
	G(d, 1, n + 2) = \langle t, r_1, r_2, \ldots, r_{n + 1} \rangle,
\]
where $t = r_{{e_1}, \zeta_d}$, $r_i = r_{e_i - e_{i + 1}, -1}$ $(i = 1, 2, \ldots, n + 1)$, and  $e_1, e_2, \ldots, e_{n + 2}$ form an orthonormal basis. 

\begin{example}
	For example, we explicitly describe the generators of $G(d,1,4)$ using the standard hermitian inner product on ${\mathbb C}^4$. For a vector $v = {}^t\!(x,y,z,w) \in {\mathbb C}^4$, 
	\begin{align*}
		t (v) = r_{{e_1}, \zeta_d} (v)  & = v - (1 - \zeta_d) \frac{(e_1, v)}{(e_1, e_1)} e_1 \\
			& = 
				\begin{pmatrix}
					x \\ y \\ z \\ w
				\end{pmatrix}
				- (1 - \zeta_d) \frac{x}{1} 
				\begin{pmatrix}
					1 \\ 0 \\ 0 \\ 0
				\end{pmatrix}  
				= 
				\begin{pmatrix}
					\zeta_d x \\ y \\ z \\ w
				\end{pmatrix}. 
	\end{align*}
	Namely $t$ is represented by the matrix $\begin{pmatrix} \zeta_d & 0 & 0 & 0 \\ 0 & 1 & 0 & 0 \\ 0 & 0 & 1 & 0 \\ 0 & 0 & 0 & 1 \end{pmatrix}.$ On the other hand, 
	\begin{align*}
		r_1 (v) & = r_{e_1 - e_{2}, -1} (v) = v - (1 - (-1)) \frac{(e_1-e_2, v)}{(e_1-e_2, e_1-e_2)} (e_1 -e_2) \\
		& = \begin{pmatrix} x \\ y \\ z \\ w \end{pmatrix} - 2 \frac{x - y}{2} \begin{pmatrix} 1 \\ -1 \\ 0 \\ 0 \end{pmatrix} = \begin{pmatrix} y \\ x \\ z \\ w \end{pmatrix}. 
	\end{align*}
	Namely $r_1$ is represented by the matrix 
	$\begin{pmatrix} 
		0 & 1 & 0 & 0 \\ 
		1 & 0 & 0 & 0 \\ 
		0 & 0 & 1 & 0 \\ 
		0 & 0 & 0 & 1
	\end{pmatrix}.$ 
	By a similar argument, $r_2$ is represented by 
	$\begin{pmatrix}
		1 & 0 & 0 & 0 \\ 
		0 & 0 & 1 & 0 \\ 
		0 & 1 & 0 & 0 \\ 
		0 & 0 & 0 & 1
	\end{pmatrix}$, $r_3$ is represented 
	by 
	$\begin{pmatrix} 
		1 & 0 & 0 & 0 \\ 
		0 & 1 & 0 & 0 \\ 
		0 & 0 & 0 & 1 \\ 
		0 & 0 & 1 & 0
	\end{pmatrix}$. That is, $G(d,1,4)$ is generated by these four matrices. Clearly, $x_0^d + x_1^d + x_2^d + x_3^d$ is an invariant polynomial of $G(d,1,4)$. 
\end{example}

\begin{lemma}
	Let $Z$ be the center of $G(d, 1, n + 2)$. Then
	\[
		Z= \left\langle 
		\begin{pmatrix} 
  			\zeta_d & 0 & \dots  & 0 \\
  			0 & \zeta_d & \dots  & 0 \\
  			\vdots & \vdots & \ddots & \vdots \\
  			0 & 0 & \dots  & \zeta_d
		\end{pmatrix}
	  	\right\rangle \cong {\mathbb Z}_d. 
	\]
\end{lemma}

Therefore, we have the natural projection
\[
	{\rm GL}(n + 2, {\mathbb C}) \supset G(d, 1, n + 2) \rightarrow G(d, 1, n + 2) / Z \subset {\rm PGL} (n + 2, {\mathbb C}) \cong {\rm Aut} ({\mathbb P}^{n + 1}).
\]
In view of the above theorem, we obtain
\[
	G(d, 1, n + 2) / Z \cong {\rm Aut} (F_d^n). 
\]
Hence, by counting reflections in $G(d,1,n+2)$, we can determine all Galois and quasi-Galois points for Fermat varieties.

By \cite{l-t}, we have the following proposition. 
\begin{proposition}[{\cite[Lemma 2.8]{l-t}}]\label{l-t}
	The element $r \in G(d, 1, n + 2)$ is a reflection if and only if $r$ has one of the following forms:
 	\begin{enumerate}
  		\item[{\rm (i)}] For some $i$ and for some $d$-th root of unity $\zeta \ne 1$, $r = r_{e_i, \zeta}$. 
  		\item[{\rm (ii)}] For some $i \ne j$ and some $d$-th root of unity $\zeta$, $r = r_{e_i - \zeta e_j, -1}$. 
 	\end{enumerate}
\end{proposition}

In particular, we see that for type ${\rm (i)}$, the order of $r$ is $d$ and the number of reflections is $(n + 2) (d - 1)$. On the other hand, for type ${\rm (ii)}$, the order of $r$ is 2 and the number of reflections is  
\[
	d \binom{n+2}{2}= \frac{d(n+2)(n+1)}{2}.
\]

Here we should remark that for type ${\rm (i)}$, $d - 1$ reflections 
\[
	\begin{pmatrix} 
		\zeta^i & 0 & \dots  & 0 \\
		0 & 1 & \dots  & 0 \\
  		\vdots & \vdots & \ddots & \vdots \\
  		0 & 0 & \dots  & 1
	\end{pmatrix} \, (i = 1,2,\ldots, d-1)
\]
belong to the same Galois point $(1 : 0 : \cdots : 0)$. Hence, although there exist $(n + 2) (d - 1)$ reflections, we obtain exactly $(n + 2) (d - 1)/ (d - 1) = n + 2$ Galois points, each of whose Galois groups is isomorphic to the cyclic group of order $d$. On the other hand, type ${\rm (ii)}$ reflections correspond to quasi-Galois points satisfying $|G[P]| = 2$. Thus we complete the proofs of Theorem~\ref{thm:FermatGalois} and Remark~\ref{rem:d=3_Fermat}. 

Applying Proposition~\ref{l-t}, we obtain the following result, which recovers our previous results.

\begin{theorem}[\cite{fmt1}, \cite{m-y1}, \cite{m-y2}, \cite{yoshi1}] 
	For the Fermat curve $F_d^1: x_0^d + x_1^d + x_2^d = 0$, 
 	\begin{enumerate}
  		\item[$(1)$] $F_d^1$ ($d \geq 3$) has three Galois points. The Galois group $G_P$ is isomorphic to ${\mathbb Z}_d$ when $P$ is a Galois point.  
  		\item[$(2)$] $F_d^1$ ($d\geq 4$) has $3 d$ quasi-Galois points of order two (i.e., $G[P] \cong {\mathbb Z}_2$). 
 	\end{enumerate}
\end{theorem}

\begin{remark}
	Applying Proposition \ref{l-t}, we can obtain concrete coordinates of Galois and quasi-Galois points for $F_d^n$. Indeed, suppose that $\{e_1, e_2, \ldots, e_{n + 2}\}$ is the standard basis. Then the quasi-Galois points are given by the points in ${\mathbb P}^{n+1}$ which correspond to $e_i - \zeta e_j = (0, \ldots, 1,  \cdots,  -\zeta,  \cdots, 0) \in {\mathbb C}^{n + 2}$.  
\end{remark}

\section*{Acknowledgments}
The authors are grateful to Professor Satoru Fukasawa for many valuable discussions and helpful suggestions during the development of this work. 
The first author was partially supported by JSPS KAKENHI Grant Number JP18K03230. 
The second author was partially supported by JSPS KAKENHI Grant Number JP25K06930.


\end{document}